\documentclass[draft]{amsart}
\usepackage{amsmath}
\usepackage{amssymb}
\usepackage[all]{xy}

\newtheorem{thm}{Theorem}
\newtheorem{lem}[thm]{Lemma}
\newtheorem{prop}[thm]{Proposition}
\newtheorem{cor}[thm]{Corollary}

\theoremstyle{remark}

\newcommand{\cv}{\mathbf{C}}
\newcommand{\pv}{\mathbf{P}}
\newcommand{\zv}{\mathbf{Z}}

\numberwithin{thm}{section}

\begin{document}


\title{Latt\`es Maps on $\pv^2$}

\author{Feng Rong}

\address{Department of Mathematics, Shanghai Jiao Tong University, 800 Dong Chuan Road, Shanghai, 200240, P.R. China}
\email{frong@syr.edu}

\maketitle

\textit{Abstract}

In [J. Milnor, On Latt\`es maps, Dynamics on the Riemann sphere, European Mathematical Society, Z\"urich, 2006, 9-43], Milnor gave a classification of Latt\`es maps on $\pv^1$. In this paper, we give a classification of Latt\`es maps on $\pv^2$.

\textit{R\'esum\'e}

Dans [J. Milnor, On Latt\`es maps, Dynamics on the Riemann sphere, European Mathematical Society, Z\"urich, 2006, 9-43], Milnor a donn\'e une classification des applications de Latt\`es sur $\pv^1$. Dans cet article, nous donnons une classification des applications de Latt\`es sur $\pv^2$.
 
\section{Introduction}\label{S:Intro}

A \textit{Latt\`es} map on $\pv^k$ is, by definition, a holomorphic map $f$ on $\pv^k$ such that the following diagram commutes:
$$\xymatrix{
\ E^k\  \ar[d]_\Psi \ar[r]^L & \ E^k\  \ar[d]^\Psi \\
\ \pv^k\  \ar[r]^f & \ \pv^k, }$$
where $E^k$ is a complex torus of dimension $k$, $L$ is an affine map on $E^k$ and $\Psi$ is a holomorphic map from $E^k$ onto $\pv^k$.

Latt\`es maps on $\pv^k$ have very interesting and unique dynamical properties, which have been studied by many authors (see e.g. \cite{Z:Lattes}, \cite{BL:Lattes1}, \cite{BL:Lattes2}, \cite{BD:Lattes}, \cite{Di:Lattes}, \cite{Du:Lattes}). In this note, however, we focus on the classification of Latt\`es maps on $\pv^2$, in the spirit of \cite{M:Lattes} which studies such maps on $\pv^1$.

To state our main result, we need a couple of definitions.

Let $g$ be a holomorphic map on $\pv^1$ and $f$ be a holomorphic map on $\pv^2$. If there exists a holomorphic map $\pi:\pv^1\times\pv^1\rightarrow\pv^2$ such that $f\circ\pi=\pi\circ(g,g)$, then we call $f$ the \textit{square map} of $g$.

An \textit{algebraic web} is given by a reduced curve $C\subset \check{\pv}^2$, where $\check{\pv}^2$ is the dual projective plane consisting of lines in $\pv^2$. The web is invariant for a holomorphic map $f$ on $\pv^2$ if every line in $\pv^2$ belonging to $C$ is mapped to another such line. A detailed study of such invariant webs has been carried out in \cite{DJ:Webs}.

The goal of this paper is to prove the following

\begin{thm}\label{T:Main}
If $f$ is a Latt\`es map on $\pv^2$, then either $f$ or a suitable iteration of $f$ is one of the following:\\
$(i)$ a square map of a Latt\`es map on $\pv^1$;\\
$(ii)$ a holomorphic map preserving an algebraic web associated to a smooth cubic.
\end{thm}

The paper is organized as follows. In Section \ref{S:Crystal}, we briefly review all possible maps $\Psi:E^2\rightarrow \pv^2$. In Section \ref{S:Affine}, we give a complete list of all possible types of affine maps on $E^2$ which induce Latt\`es maps on $\pv^2$. We then prove Theorem \ref{T:Main} in Section \ref{S:Class}. In Section \ref{S:Elliptic}, we recall some basic facts about elliptic functions and the notion of ``orbifold". Finally in Section \ref{S:Examples}, we give examples of typical types of Latt\`es maps on $\pv^2$.

\section{Complex Crystallographic Groups}\label{S:Crystal}

Let $E(n)$ be the complex motion group acting on $\cv^n$. A \textit{complex crystallographic group} is, by definition, a discrete subgroup of $E(n)$ with compact quotient. Let $U(n)$ be the unitary group of size $n$.

For $A\in U(2)$ and $r\in \cv^2$, let $(A|r)\in E(2)$ denote the transformation: $z\rightarrow Az+r$. For a two dimensional complex crystallographic group $\Gamma$, $R:=\{r;\ (1|r)\in\Gamma\}$ and $G:=\{A;\ (A|r)\in\Gamma\}$ are called the lattice and the point group of $\Gamma$, respectively. If $\Gamma$ has the representation $\{(A|r);\ A\in G,\ r\in R\}$, then $\Gamma$ is called the semidirect product $G\ltimes R$ of the lattice and the point group.

Let $G(m,p,2)\subset U(2)$ denote the group generated by
$$\left( \begin{array}{cc} & 1 \\ 1 & \end{array} \right),\ \left( \begin{array}{cc} & \theta \\ \theta^{-1} & \end{array} \right) \hbox{ and } \left( \begin{array}{cc} \theta^p & \\ & 1 \end{array} \right),\ \ \theta=e^{2\pi i/m}.$$

We have the following theorem from \cite{KTY:II}.

\begin{thm}\cite[Theorem 1]{KTY:II}
Every two dimensional complex crystallographic group $\Gamma$ with $\cv^2/\Gamma\cong \pv^2$ is conjugate, in the affine transformation group, to one of the following six groups,
\begin{align*}
(2,1)_0 & :=G(2,1,2)\ltimes L^2(\tau),\\
(3,1)_0 & :=G(3,1,2)\ltimes L^2(\zeta),\\
(4,1)_0 & :=G(4,1,2)\ltimes L^2(i),\\
(6,1)_0 & :=G(6,1,2)\ltimes L^2(\zeta),\\
(4,2)_1 & :=G(4,2,2)\ltimes \left\{L^2(i)+\zv\frac{1+i}{2}\left( \begin{array}{c} 1\\ 1 \end{array} \right)\right\},\\
(3,3)_0 & :=G(3,3,2)\ltimes \left\{L(\tau)\left( \begin{array}{c} -1 \\ 1 \end{array} \right) + L(\tau)\left( \begin{array}{c} \zeta^2 \\ \zeta \end{array} \right)\right\},
\end{align*}
where $L(\tau)=\zv+\tau\zv$, $L^2(\tau)=L(\tau)\left( \begin{array}{c} 1 \\ 0 \end{array} \right) + L(\tau)\left( \begin{array}{c} 0 \\ 1 \end{array} \right)$, \hbox{Im} $\tau>0$, $\zeta=e^{2\pi i/6}$.
\end{thm}

A two-dimensional complex torus $E^2$ is given as a quotient $\cv^2/\Lambda$, with $\Lambda\subset \cv^2$ a lattice. An easy consequence of the above theorem is the following

\begin{cor}\label{C:Quotient}
If $(\Lambda,G)$ is a pair such that $E^2=\cv^2/\Lambda$ and $E^2/G\cong \pv^2$, then, up to conjugation, it is one of the following
\begin{align*}
\Lambda=L^2(\tau), & \ G=G_1:=G(2,1,2),\\
\Lambda=L^2(\zeta), & \ G=G_2:=G(3,1,2),\\
\Lambda=L^2(i), & \ G=G_3:=G(4,1,2),\\
\Lambda=L^2(\zeta), & \ G=G_4:=G(6,1,2),\\
\Lambda=L^2(i), & \ G=G_5:=G(4,2,2)\ltimes \zv\frac{1+i}{2}\left( \begin{array}{c} 1\\ 1 \end{array} \right),\\
\Lambda=\Lambda_6:=L(\tau)\left( \begin{array}{c} -1 \\ 1 \end{array} \right) + L(\tau)\left( \begin{array}{c} \zeta^2 \\ \zeta \end{array} \right), & \ G=G_6:=G(3,3,2).
\end{align*}
\end{cor}

\section{Affine Maps}\label{S:Affine}

In this section, we determine all affine maps $L$ on $E^2$ which induce Latt\`es maps on $\pv^2$.

Write $L=(A|r)$, where
$$A=\left(\begin{array}{cc} a & b \\ c & d \\ \end{array}\right),\ \ \ r=\left(\begin{array}{c} e \\ f \end{array}\right),$$
with $a,b,c,d,e,f\in \cv$.

Let $(\Lambda,G)$ be one of the pairs in Corollary \ref{C:Quotient}. If $L$ induces a Latt\`es map $f$ on $\pv^2$, then we have $f\circ\Psi=\Psi\circ L$. Let $\pi:\cv^2\rightarrow E^2$ be the quotient map and set $\psi=\Psi\circ\pi$.  Consider $l=(A|r)$ as an affine map on $\cv^2$, we get $f\circ\psi=\psi\circ l$. For $z\in \cv^2$ and $\lambda\in \Lambda$ we then have
$$\psi\circ(Az+r)=f\circ\psi(z)=f\circ\psi(z+\lambda)=\psi\circ(A(z+\lambda)+r).$$
This shows that
\begin{equation}\label{E:A1}
A\Lambda\subset \Lambda.
\end{equation}

Since $L$ carries a small domain of volume $V$ to a domain of volume $|\hbox{det} A|^2V$, it follows that the map $L$ has degree $|\hbox{det} A|^2$. Every non-constant holomorphic map $f$ on $\pv^2$ has the form $[P:Q:R]$, where $P,Q,R$ are homogeneous polynomials of degree $d_f$ (see \cite{FS:DynamicsI}). We call $d_f$ the algebraic degree of $f$ and the (topological) degree of $f$ is $d_f^2$. Since $f\circ\Psi=\Psi\circ L$, we get $d_f^2=|\hbox{det} A|^2$, and thus
\begin{equation}\label{E:A2}
|\hbox{det} A|=d_f.
\end{equation}

By definition, an affine map $L$ on $E^2=\cv^2/\Lambda$ induces a Latt\`es map $f$ on $\pv^2\cong E^2/G$ if and only if
\begin{equation}\label{E:L/G}
\hbox{for any } g\in G, \hbox{ there exists } h\in G \hbox{ such that } L\circ g\equiv h\circ L \hbox{ mod } \Lambda.
\end{equation}
For $L=(A|r)$, $g=(B|r^\prime)$ and $h=(C|r^{\prime\prime})$, we have $L\circ g=(AB|Ar^\prime+r)$ and $h\circ L=(CA|Cr+r^{\prime\prime})$. Therefore $L\circ g\equiv h\circ L \hbox{ mod } \Lambda$ is equivalent to $AB=CA$ and $Ar^\prime+r\equiv Cr+r^{\prime\prime} \hbox{ mod } \Lambda$. Here, we make two simple yet important observations:\\
1. condition (\ref{E:L/G}) is satisfied if and only if it is satisfied for all generators of $G$;\\
2. if $g=(B|r^\prime)$ and $h=(C|r^{\prime\prime})$ satisfy $L\circ g\equiv h\circ L \hbox{ mod } \Lambda$, then $B$ and $C$ are of the same order.

The rest of this section is devoted to the investigation of condition (\ref{E:L/G}) for each pair $(\Lambda,G)$ in Corollary \ref{C:Quotient}.

We start with the easiest case $(\Lambda,G_6)$, for which we have the following

\begin{prop}\label{P:G6}
If $L=(A|r)$ induces a Latt\`es map in the case $(\Lambda,G_6)$, then $L$ takes one of the following six forms:

$1) \ A=a\left( \begin{array}{cc} 1 & \\ & 1 \end{array} \right),\ r=\left( \begin{array}{c} e \\ f \end{array} \right)\equiv \left( \begin{array}{c} f \\ e \end{array} \right)\equiv \left( \begin{array}{c} \zeta^2 f \\ \zeta^{-2}e \end{array} \right) \hbox{ mod } \Lambda_6$;

$2) \ A=a\left( \begin{array}{cc} 1 & \\ & \zeta^{-2} \end{array} \right),\ r=\left( \begin{array}{c} e \\ f \end{array} \right)\equiv \left( \begin{array}{c} \zeta^{-2}f \\ \zeta^2 e \end{array} \right)\equiv \left( \begin{array}{c} \zeta^2 f \\ \zeta^{-2}e \end{array} \right) \hbox{ mod } \Lambda_6$;

$3) \ A=a\left( \begin{array}{cc} 1 & \\ & \zeta^2 \end{array} \right),\ r=\left( \begin{array}{c} e \\ f \end{array} \right)\equiv \left( \begin{array}{c} f \\ e \end{array} \right)\equiv \left( \begin{array}{c} \zeta^{-2}f \\ \zeta^2 e \end{array} \right) \hbox{ mod } \Lambda_6$;

$4) \ A=b\left( \begin{array}{cc} & 1 \\ 1 & \end{array} \right),\ r=\left( \begin{array}{c} e \\ f \end{array} \right)\equiv \left( \begin{array}{c} f \\ e \end{array} \right)\equiv \left( \begin{array}{c} \zeta^{-2}f \\ \zeta^2 e \end{array} \right) \hbox{ mod } \Lambda_6$;

$5) \ A=b\left( \begin{array}{cc} & 1 \\ \zeta^{-2} & \end{array} \right),\ r=\left( \begin{array}{c} e \\ f \end{array} \right)\equiv \left( \begin{array}{c} f \\ e \end{array} \right)\equiv \left( \begin{array}{c} \zeta^2 f \\ \zeta^{-2}e \end{array} \right) \hbox{ mod } \Lambda_6$;

$6) \ A=b\left( \begin{array}{cc} & 1 \\ \zeta^2 & \end{array} \right),\ r=\left( \begin{array}{c} e \\ f \end{array} \right)\equiv \left( \begin{array}{c} \zeta^{-2}f \\ \zeta^2 e \end{array} \right)\equiv \left( \begin{array}{c} \zeta^2 f \\ \zeta^{-2}e \end{array} \right) \hbox{ mod } \Lambda_6$.
\end{prop}
\begin{proof}
The group $G_6$ is generated by
$$u=\left( \begin{array}{cc} & 1 \\ 1 & \end{array} \right) \ \hbox{ and } \ v=\left( \begin{array}{cc} & \theta \\ \theta^{-1} & \end{array} \right),\ \ \theta=\zeta^2,$$
where both $u$ and $v$ are of order 2. The other elements of $G_6$, other than the identity, are
$$uv=\left( \begin{array}{cc} \theta^{-1} & \\ & \theta \end{array} \right),\ vu=\left( \begin{array}{cc} \theta & \\ & \theta^{-1} \end{array} \right) \ \hbox{ and }\  uvu=\left( \begin{array}{cc} & \theta^{-1} \\ \theta & \end{array} \right),$$
where $uvu$ is of order 2 and both $uv$ and $vu$ are of order 3.

For $L=(A|r)$, we have
$$L\circ u=\left( \begin{array}{cc} b & a \\ d & c \end{array} \right.\left| \begin{array}{c} e \\ f \end{array} \right),\ L\circ v=\left( \begin{array}{cc} \theta^{-1}b & \theta a \\ \theta^{-1}d & \theta c \end{array} \right.\left| \begin{array}{c} e \\ f \end{array} \right),$$
and
$$u\circ L=\left( \begin{array}{cc} c & d \\ a & b \end{array} \right.\left| \begin{array}{c} f \\ e \end{array} \right),\ v\circ L=\left( \begin{array}{cc} \theta c & \theta d \\ \theta^{-1}a & \theta^{-1}b \end{array} \right.\left| \begin{array}{c} \theta f \\ \theta^{-1}e \end{array} \right),$$
$$uvu\circ L=\left( \begin{array}{cc} \theta^{-1}c & \theta^{-1}d \\ \theta a & \theta b \end{array} \right.\left| \begin{array}{c} \theta^{-1}f \\ \theta e \end{array} \right).$$

By observations 1 and 2, we have the following six possibilities:

1) $L\circ u\equiv u\circ L \hbox{ mod } \Lambda$ and $L\circ v\equiv v\circ L \hbox{ mod } \Lambda$;

2) $L\circ u\equiv v\circ L \hbox{ mod } \Lambda$ and $L\circ v\equiv uvu\circ L \hbox{ mod } \Lambda$;

3) $L\circ u\equiv uvu\circ L \hbox{ mod } \Lambda$ and $L\circ v\equiv u\circ L \hbox{ mod } \Lambda$;

4) $L\circ u\equiv u\circ L \hbox{ mod } \Lambda$ and $L\circ v\equiv uvu\circ L \hbox{ mod } \Lambda$;

5) $L\circ u\equiv v\circ L \hbox{ mod } \Lambda$ and $L\circ v\equiv u\circ L \hbox{ mod } \Lambda$;

6) $L\circ u\equiv uvu\circ L \hbox{ mod } \Lambda$ and $L\circ v\equiv v\circ L \hbox{ mod } \Lambda$.
\end{proof}

The case $(\Lambda,G_2)$ is similar to the case $(\Lambda,G_6)$. We have the following

\begin{prop}\label{P:G2}
If $L=(A|r)$ induces a Latt\`es map in the case $(\Lambda,G_2)$, then $A=aB$, where $B$ takes one of the following six forms:
$$\begin{array}{lll}
1)\ \left( \begin{array}{cc} 1 & \\ & 1 \end{array} \right), & 2)\ \left( \begin{array}{cc} 1 & \\ & \zeta^{-2} \end{array} \right), & 3)\ \left( \begin{array}{cc} 1 & \\ & \zeta^2 \end{array} \right),\\
4)\ \left( \begin{array}{cc} & 1 \\ 1 & \end{array} \right), & 5)\ \left( \begin{array}{cc} & 1 \\ \zeta^{-2} & \end{array} \right), & 6)\ \left( \begin{array}{cc} & 1 \\ \zeta^2 & \end{array} \right),
\end{array}$$
and
$$r\equiv e\left( \begin{array}{c} 1 \\ 1 \end{array} \right) \hbox{ mod } L^2(\zeta),\ \ e\equiv \zeta^2e \hbox{ mod } L(\zeta).$$
\end{prop}
\begin{proof}
The group $G_2$ is generated by
$$u=\left( \begin{array}{cc} & 1 \\ 1 & \end{array} \right),\ v=\left( \begin{array}{cc} & \theta \\ \theta^{-1} & \end{array} \right) \ \hbox{ and } \ w=\left( \begin{array}{cc} \theta & \\ & 1 \end{array} \right),\ \ \theta=\zeta^2,$$
where both $u$ and $v$ are of order 2 and $w$ is of order 3. One readily checks that $uvu$ is the only other element of order 2 of $G_2$. As in the proof of Proposition \ref{P:G6}, we have the following six possibilities:

1) $L\circ u\equiv u\circ L \hbox{ mod } \Lambda$, $L\circ v\equiv v\circ L \hbox{ mod } \Lambda$ ($\Rightarrow L\circ w\equiv w\circ L \hbox{ mod } \Lambda$);

2) $L\circ u\equiv v\circ L \hbox{ mod } \Lambda$, $L\circ v\equiv uvu\circ L \hbox{ mod } \Lambda$ ($\Rightarrow L\circ w\equiv w\circ L \hbox{ mod } \Lambda$);

3) $L\circ u\equiv uvu\circ L \hbox{ mod } \Lambda$, $L\circ v\equiv u\circ L \hbox{ mod } \Lambda$ ($\Rightarrow L\circ w\equiv w\circ L \hbox{ mod } \Lambda$);

4) $L\circ u\equiv u\circ L \hbox{ mod } \Lambda$, $L\circ v\equiv uvu\circ L \hbox{ mod } \Lambda$ ($\Rightarrow L\circ w\equiv uvw\circ L \hbox{ mod } \Lambda$);

5) $L\circ u\equiv v\circ L \hbox{ mod } \Lambda$, $L\circ v\equiv u\circ L \hbox{ mod } \Lambda$ ($\Rightarrow L\circ w\equiv uvw\circ L \hbox{ mod } \Lambda$);

6) $L\circ u\equiv uvu\circ L \hbox{ mod } \Lambda$, $L\circ v\equiv v\circ L \hbox{ mod } \Lambda$ ($\Rightarrow L\circ w\equiv uvw\circ L \hbox{ mod } \Lambda$).
\end{proof}

The next proposition deals with the case $(\Lambda,G_1)$.

\begin{prop}\label{P:G1}
If $L=(A|r)$ induces a Latt\`es map in the case $(\Lambda,G_1)$, then $A=aB$, where $B$ takes one of the following eight forms:
$$\begin{array}{llll}
1)\ \left( \begin{array}{cc} 1 & 1 \\ 1 & -1 \end{array} \right), & 2)\ \left( \begin{array}{cc} 1 & -1 \\ 1 & 1 \end{array} \right), & 3)\ \left( \begin{array}{cc} 1 & 1 \\ -1 & 1 \end{array} \right), & 4)\ \left( \begin{array}{cc} 1 & -1 \\ -1 & -1 \end{array} \right),\\
5)\ \left( \begin{array}{cc} 1 & \\ & 1 \end{array} \right), & 6)\ \left( \begin{array}{cc} 1 & \\ & -1 \end{array} \right), & 7)\ \left( \begin{array}{cc} & 1 \\ 1 & \end{array} \right), & 8)\ \left( \begin{array}{cc} & 1 \\ -1 & \end{array} \right),
\end{array}$$
and
$$r\equiv e\left( \begin{array}{c} 1 \\ 1 \end{array} \right) \hbox{ mod } L^2(\tau),\ \ e\equiv -e \hbox{ mod } L(\tau).$$
\end{prop}
\begin{proof}
The group $G_1$ is generated by
$$u=\left( \begin{array}{cc} & 1 \\ 1 & \end{array} \right),\ v=\left( \begin{array}{cc} & -1 \\ -1 & \end{array} \right) \ \hbox{ and } \ w=\left( \begin{array}{cc} -1 & \\ & 1 \end{array} \right),$$
and each element, other than the identity, of $G_1$ is of order 2.

A careful examination shows that we have the following eight possibilities:

1) $L\circ w\equiv v\circ L \hbox{ mod } \Lambda$, $L\circ u\equiv uvw\circ L \hbox{ mod } \Lambda$ ($\Rightarrow L\circ v\equiv w\circ L \hbox{ mod } \Lambda$);

2) $L\circ w\equiv v\circ L \hbox{ mod } \Lambda$, $L\circ u\equiv w\circ L \hbox{ mod } \Lambda$ ($\Rightarrow L\circ v\equiv uvw\circ L \hbox{ mod } \Lambda$);

3) $L\circ w\equiv u\circ L \hbox{ mod } \Lambda$, $L\circ u\equiv uvw\circ L \hbox{ mod } \Lambda$ ($\Rightarrow L\circ v\equiv w\circ L \hbox{ mod } \Lambda$);

4) $L\circ w\equiv u\circ L \hbox{ mod } \Lambda$, $L\circ u\equiv w\circ L \hbox{ mod } \Lambda$ ($\Rightarrow L\circ v\equiv uvw\circ L \hbox{ mod } \Lambda$);

5) $L\circ w\equiv w\circ L \hbox{ mod } \Lambda$, $L\circ u\equiv u\circ L \hbox{ mod } \Lambda$ ($\Rightarrow L\circ v\equiv v\circ L \hbox{ mod } \Lambda$);

6) $L\circ w\equiv w\circ L \hbox{ mod } \Lambda$, $L\circ u\equiv v\circ L \hbox{ mod } \Lambda$ ($\Rightarrow L\circ v\equiv u\circ L \hbox{ mod } \Lambda$);

7) $L\circ w\equiv uvw\circ L \hbox{ mod } \Lambda$, $L\circ u\equiv u\circ L \hbox{ mod } \Lambda$ ($\Rightarrow L\circ v\equiv v\circ L \hbox{ mod } \Lambda$);

8) $L\circ w\equiv uvw\circ L \hbox{ mod } \Lambda$, $L\circ u\equiv v\circ L \hbox{ mod } \Lambda$ ($\Rightarrow L\circ v\equiv u\circ L \hbox{ mod } \Lambda$).
\end{proof}

Next, we study the case $(\Lambda,G_3)$.

\begin{prop}\label{P:G3}
If $L=(A|r)$ induces a Latt\`es map in the case $(\Lambda,G_3)$, then $A=aB$, where $B$ takes one of the following eight forms:
$$\begin{array}{llll}
1)\ \left( \begin{array}{cc} 1 & \\ & 1 \end{array} \right), & 2)\ \left( \begin{array}{cc} 1 & \\ & -1 \end{array} \right), & 3)\ \left( \begin{array}{cc} & 1 \\ 1 & \end{array} \right), & 4)\ \left( \begin{array}{cc} & 1 \\ -1 & \end{array} \right),\\
5)\ \left( \begin{array}{cc} 1 & \\ & -i \end{array} \right), & 6)\ \left( \begin{array}{cc} 1 & \\ & i \end{array} \right), & 7)\ \left( \begin{array}{cc} & 1 \\ -i & \end{array} \right), & 8)\ \left( \begin{array}{cc} & 1 \\ i & \end{array} \right),
\end{array}$$
and
$$r\equiv e\left( \begin{array}{c} 1 \\ 1 \end{array} \right) \hbox{ mod } L^2(i),\ \ e\equiv ie \hbox{ mod } L(i).$$
\end{prop}
\begin{proof}
The group $G_3$ is generated by
$$u=\left( \begin{array}{cc} & 1 \\ 1 & \end{array} \right),\ s=\left( \begin{array}{cc} & i \\ -i & \end{array} \right) \ \hbox{ and } \ t=\left( \begin{array}{cc} i & \\ & 1 \end{array} \right),$$
where both $u$ and $s$ are of order 2 and $t$ is of order 4. One readily checks that $G_1$ is a subgroup of $G_3$ and the only elements of order 2 of $G_3$, other than those of $G_1$, are $s$ and $usu$.

Every element of $G_3$ is of the form
$$\left( \begin{array}{cc} \alpha & \\ & \beta \end{array} \right)\ \hbox{ or }\ \left( \begin{array}{cc} & \alpha \\ \beta & \end{array} \right),$$
with $\alpha\neq 0$ and $\beta\neq 0$. It is then easy to check that either $b=c=0$ or $a=d=0$, if there exists $\tilde{t}\in G_3$ such that $At=\tilde{t}A$.

A careful examination then shows that we have the following eight possibilities:

1) $L\circ u\equiv u\circ L \hbox{ mod } \Lambda$, $L\circ s\equiv s\circ L \hbox{ mod } \Lambda$ ($\Rightarrow L\circ t\equiv t\circ L \hbox{ mod } \Lambda$);

2) $L\circ u\equiv v\circ L \hbox{ mod } \Lambda$, $L\circ s\equiv usu\circ L \hbox{ mod } \Lambda$ ($\Rightarrow L\circ t\equiv t\circ L \hbox{ mod } \Lambda$);

3) $L\circ u\equiv u\circ L \hbox{ mod } \Lambda$, $L\circ s\equiv usu\circ L \hbox{ mod } \Lambda$ ($\Rightarrow L\circ t\equiv ust\circ L \hbox{ mod } \Lambda$);

4) $L\circ u\equiv v\circ L \hbox{ mod } \Lambda$, $L\circ s\equiv s\circ L \hbox{ mod } \Lambda$ ($\Rightarrow L\circ t\equiv ust\circ L \hbox{ mod } \Lambda$);

5) $L\circ u\equiv s\circ L \hbox{ mod } \Lambda$, $L\circ s\equiv v\circ L \hbox{ mod } \Lambda$ ($\Rightarrow L\circ t\equiv t\circ L \hbox{ mod } \Lambda$);

6) $L\circ u\equiv usu\circ L \hbox{ mod } \Lambda$, $L\circ s\equiv u\circ L \hbox{ mod } \Lambda$ ($\Rightarrow L\circ t\equiv t\circ L \hbox{ mod } \Lambda$);

7) $L\circ u\equiv s\circ L \hbox{ mod } \Lambda$, $L\circ s\equiv u\circ L \hbox{ mod } \Lambda$ ($\Rightarrow L\circ t\equiv ust\circ L \hbox{ mod } \Lambda$);

8) $L\circ u\equiv usu\circ L \hbox{ mod } \Lambda$, $L\circ s\equiv v\circ L \hbox{ mod } \Lambda$ ($\Rightarrow L\circ t\equiv ust\circ L \hbox{ mod } \Lambda$).
\end{proof}

The case $(\Lambda,G_4)$ is similar to the case $(\Lambda,G_3)$. We have the following

\begin{prop}\label{P:G4}
If $L=(A|r)$ induces a Latt\`es map in the case $(\Lambda,G_3)$, then $A=aB$, where $B$ takes one of the following twelve forms:
$$\begin{array}{llll}
1)\ \left( \begin{array}{cc} 1 & \\ & 1 \end{array} \right), & 2)\ \left( \begin{array}{cc} 1 & \\ & -1 \end{array} \right), & 3)\ \left( \begin{array}{cc} & 1 \\ 1 & \end{array} \right), & 4)\ \left( \begin{array}{cc} & 1 \\ -1 & \end{array} \right),\\
5)\ \left( \begin{array}{cc} 1 & \\ & \zeta^{-1} \end{array} \right), & 6)\ \left( \begin{array}{cc} 1 & \\ & \zeta \end{array} \right), & 7)\ \left( \begin{array}{cc} & 1 \\ \zeta^{-1} & \end{array} \right), & 8)\ \left( \begin{array}{cc} & 1 \\ \zeta & \end{array} \right),\\
9)\ \left( \begin{array}{cc} 1 & \\ & \zeta^{-2} \end{array} \right), & 10)\ \left( \begin{array}{cc} 1 & \\ & \zeta^2 \end{array} \right), & 11)\ \left( \begin{array}{cc} & 1 \\ \zeta^{-2} & \end{array} \right), & 12)\ \left( \begin{array}{cc} & 1 \\ \zeta^2 & \end{array} \right),
\end{array}$$
and
$$r\equiv e\left( \begin{array}{c} 1 \\ 1 \end{array} \right) \hbox{ mod } L^2(\zeta),\ \ e\equiv \zeta e \hbox{ mod } L(\zeta).$$
\end{prop}
\begin{proof}
The group $G_4$ is generated by
$$u=\left( \begin{array}{cc} & 1 \\ 1 & \end{array} \right),\ s=\left( \begin{array}{cc} & \zeta \\ \zeta^{-1} & \end{array} \right) \ \hbox{ and } \ t=\left( \begin{array}{cc} \zeta & \\ & 1 \end{array} \right),$$
where both $u$ and $s$ are of order 2 and $t$ is of order 6. One readily checks that $G_1$ is a subgroup of $G_4$ and the only elements of order 2 of $G_4$, other than those of $G_1$, are $s$, $usu$, $sus$ and $ususu$.

As in the proof of Proposition \ref{P:G3}, we have either $b=c=0$ or $a=d=0$.

A careful examination shows that we have the following twelve possibilities:

1) $L\circ u\equiv u\circ L \hbox{ mod } \Lambda$, $L\circ s\equiv s\circ L \hbox{ mod } \Lambda$ ($\Rightarrow L\circ t\equiv t\circ L \hbox{ mod } \Lambda$);

2) $L\circ u\equiv v\circ L \hbox{ mod } \Lambda$, $L\circ s\equiv ususu\circ L \hbox{ mod } \Lambda$ ($\Rightarrow L\circ t\equiv t\circ L \hbox{ mod } \Lambda$);

3) $L\circ u\equiv u\circ L \hbox{ mod } \Lambda$, $L\circ s\equiv usu\circ L \hbox{ mod } \Lambda$ ($\Rightarrow L\circ t\equiv ust\circ L \hbox{ mod } \Lambda$);

4) $L\circ u\equiv v\circ L \hbox{ mod } \Lambda$, $L\circ s\equiv sus\circ L \hbox{ mod } \Lambda$ ($\Rightarrow L\circ t\equiv ust\circ L \hbox{ mod } \Lambda$);

5) $L\circ u\equiv s\circ L \hbox{ mod } \Lambda$, $L\circ s\equiv sus\circ L \hbox{ mod } \Lambda$ ($\Rightarrow L\circ t\equiv t\circ L \hbox{ mod } \Lambda$);

6) $L\circ u\equiv usu\circ L \hbox{ mod } \Lambda$, $L\circ s\equiv u\circ L \hbox{ mod } \Lambda$ ($\Rightarrow L\circ t\equiv t\circ L \hbox{ mod } \Lambda$);

7) $L\circ u\equiv s\circ L \hbox{ mod } \Lambda$, $L\circ s\equiv u\circ L \hbox{ mod } \Lambda$ ($\Rightarrow L\circ t\equiv ust\circ L \hbox{ mod } \Lambda$);

8) $L\circ u\equiv usu\circ L \hbox{ mod } \Lambda$, $L\circ s\equiv ususu\circ L \hbox{ mod } \Lambda$ ($\Rightarrow L\circ t\equiv ust\circ L \hbox{ mod } \Lambda$);

9) $L\circ u\equiv sus\circ L \hbox{ mod } \Lambda$, $L\circ s\equiv v\circ L \hbox{ mod } \Lambda$ ($\Rightarrow L\circ t\equiv t\circ L \hbox{ mod } \Lambda$);

10) $L\circ u\equiv ususu\circ L \hbox{ mod } \Lambda$, $L\circ s\equiv usu\circ L \hbox{ mod } \Lambda$ ($\Rightarrow L\circ t\equiv t\circ L \hbox{ mod } \Lambda$);

11) $L\circ u\equiv sus\circ L \hbox{ mod } \Lambda$, $L\circ s\equiv s\circ L \hbox{ mod } \Lambda$ ($\Rightarrow L\circ t\equiv ust\circ L \hbox{ mod } \Lambda$);

12) $L\circ u\equiv ususu\circ L \hbox{ mod } \Lambda$, $L\circ s\equiv v\circ L \hbox{ mod } \Lambda$ ($\Rightarrow L\circ t\equiv ust\circ L \hbox{ mod } \Lambda$).
\end{proof}

Finally, the next proposition deals with the hardest case $(\Lambda,G_5)$.

\begin{prop}\label{P:G5}
If $L=(A|r)$ induces a Latt\`es map in the case $(\Lambda,G_5)$, then $A=aB$, where $B$ takes one of the following twenty-four forms:
$$\begin{array}{llll}
1)\ \left( \begin{array}{cc} 1 & 1 \\ 1 & -1 \end{array} \right), & 2)\ \left( \begin{array}{cc} 1 & -1 \\ 1 & 1 \end{array} \right), & 3)\ \left( \begin{array}{cc} 1 & 1 \\ -1 & 1 \end{array} \right), & 4)\ \left( \begin{array}{cc} 1 & -1 \\ -1 & -1 \end{array} \right),\\
5)\ \left( \begin{array}{cc} 1 & -i \\ i & -1 \end{array} \right), & 6)\ \left( \begin{array}{cc} 1 & -i \\ -i & 1 \end{array} \right), & 7)\ \left( \begin{array}{cc} 1 & i \\ i & 1 \end{array} \right), & 8)\ \left( \begin{array}{cc} 1 & i \\ -i & -1 \end{array} \right),\\
9)\ \left( \begin{array}{cc} 1 & i \\ 1 & -i \end{array} \right), & 10)\ \left( \begin{array}{cc} 1 & -i \\ 1 & i \end{array} \right), & 11)\ \left( \begin{array}{cc} 1 & i \\ -1 & i \end{array} \right), & 12)\ \left( \begin{array}{cc} 1 & -i \\ -1 & -i \end{array} \right),\\
13)\ \left( \begin{array}{cc} 1 & 1 \\ i & -i \end{array} \right), & 14)\ \left( \begin{array}{cc} 1 & -1 \\ i & i \end{array} \right), & 15)\ \left( \begin{array}{cc} 1 & 1 \\ -i & i \end{array} \right), & 16)\ \left( \begin{array}{cc} 1 & -1 \\ -i & -i \end{array} \right),\\
17)\ \left( \begin{array}{cc} 1 & \\ & 1 \end{array} \right), & 18)\ \left( \begin{array}{cc} 1 & \\ & -1 \end{array} \right), & 19)\ \left( \begin{array}{cc} & 1 \\ 1 & \end{array} \right), & 20)\ \left( \begin{array}{cc} & 1 \\ -1 & \end{array} \right),\\
21)\ \left( \begin{array}{cc} 1 & \\ & -i \end{array} \right), & 22)\ \left( \begin{array}{cc} 1 & \\ & i \end{array} \right), & 23)\ \left( \begin{array}{cc} & 1 \\ -i & \end{array} \right), & 24)\ \left( \begin{array}{cc} & 1 \\ i & \end{array} \right),
\end{array}$$
and
$$r\equiv e\left( \begin{array}{c} 1 \\ 1 \end{array} \right) \hbox{ mod } L^2(i),\ \ e\equiv ie \hbox{ mod } L(i).$$
\end{prop}
\begin{proof}
The group $G_5$ is generated by
$$u=\left( \begin{array}{cc} & 1 \\ 1 & \end{array} \right),\ s=\left( \begin{array}{cc} & i \\ -i & \end{array} \right),\ w=\left( \begin{array}{cc} -1 & \\ & 1 \end{array} \right),$$
$$\hbox{and }\ \ \ \  t=\left( \begin{array}{cc} & -1 \\ -1 & \end{array} \right.\left| \frac{1+i}{2} \left( \begin{array}{c} 1 \\ 1 \end{array} \right) \right),$$
where every generator is of order 2. One readily checks that $G_1$ is a subgroup of $G_5$ and the only elements of order 2 of $G_5$, other than those of $G_1$, are $s$, $usu$, $t$ and $ut$. We will also need
$$uvt=\left( \begin{array}{cc} & 1 \\ 1 & \end{array} \right.\left| \frac{1+i}{2} \left( \begin{array}{c} -1 \\ -1 \end{array} \right) \right),\ uts=\left( \begin{array}{cc} & -i \\ i & \end{array} \right.\left| \frac{1+i}{2} \left( \begin{array}{c} 1 \\ 1 \end{array} \right) \right),$$
$$\hbox{and }\ \ \ \  ust=\left( \begin{array}{cc} & i \\ -i & \end{array} \right.\left|\frac{1+i}{2} \left( \begin{array}{c} -i \\ i \end{array} \right) \right).$$
Note that though $uvt$, $uts$ and $ust$ are not of order 2, the matrix parts of them are.

A careful examination shows that we have the following twenty-four possibilities:

1) $L\circ w\equiv v\circ L \hbox{ mod } \Lambda$, $L\circ u\equiv uvw\circ L \hbox{ mod } \Lambda$

($\Rightarrow L\circ s\equiv usu\circ L \hbox{ mod } \Lambda$, $L\circ t\equiv w\circ L \hbox{ mod } \Lambda$);

2) $L\circ w\equiv v\circ L \hbox{ mod } \Lambda$, $L\circ u\equiv w\circ L \hbox{ mod } \Lambda$

($\Rightarrow L\circ s\equiv s\circ L \hbox{ mod } \Lambda$, $L\circ t\equiv uvw\circ L \hbox{ mod } \Lambda$);

3) $L\circ w\equiv u\circ L \hbox{ mod } \Lambda$, $L\circ u\equiv uvw\circ L \hbox{ mod } \Lambda$

($\Rightarrow L\circ s\equiv s\circ L \hbox{ mod } \Lambda$, $L\circ t\equiv w\circ L \hbox{ mod } \Lambda$);

4) $L\circ w\equiv u\circ L \hbox{ mod } \Lambda$, $L\circ u\equiv w\circ L \hbox{ mod } \Lambda$

($\Rightarrow L\circ s\equiv usu\circ L \hbox{ mod } \Lambda$, $L\circ t\equiv uvw\circ L \hbox{ mod } \Lambda$);

5) $L\circ w\equiv s\circ L \hbox{ mod } \Lambda$, $L\circ u\equiv v\circ L \hbox{ mod } \Lambda$

($\Rightarrow L\circ s\equiv w\circ L \hbox{ mod } \Lambda$, $L\circ t\equiv u\circ L \hbox{ mod } \Lambda$);

6) $L\circ w\equiv usu\circ L \hbox{ mod } \Lambda$, $L\circ u\equiv u\circ L \hbox{ mod } \Lambda$

($\Rightarrow L\circ s\equiv w\circ L \hbox{ mod } \Lambda$, $L\circ t\equiv v\circ L \hbox{ mod } \Lambda$);

7) $L\circ w\equiv s\circ L \hbox{ mod } \Lambda$, $L\circ u\equiv u\circ L \hbox{ mod } \Lambda$

($\Rightarrow L\circ s\equiv uvw\circ L \hbox{ mod } \Lambda$, $L\circ t\equiv v\circ L \hbox{ mod } \Lambda$);

8) $L\circ w\equiv usu\circ L \hbox{ mod } \Lambda$, $L\circ u\equiv v\circ L \hbox{ mod } \Lambda$

($\Rightarrow L\circ s\equiv uvw\circ L \hbox{ mod } \Lambda$, $L\circ t\equiv u\circ L \hbox{ mod } \Lambda$);

9) $L\circ w\equiv v\circ L \hbox{ mod } \Lambda$, $L\circ u\equiv s\circ L \hbox{ mod } \Lambda$

($\Rightarrow L\circ s\equiv uvw\circ L \hbox{ mod } \Lambda$, $L\circ t\equiv usu\circ L \hbox{ mod } \Lambda$);

10) $L\circ w\equiv v\circ L \hbox{ mod } \Lambda$, $L\circ u\equiv usu\circ L \hbox{ mod } \Lambda$

($\Rightarrow L\circ s\equiv w\circ L \hbox{ mod } \Lambda$, $L\circ t\equiv s\circ L \hbox{ mod } \Lambda$);

11) $L\circ w\equiv u\circ L \hbox{ mod } \Lambda$, $L\circ u\equiv usu\circ L \hbox{ mod } \Lambda$

($\Rightarrow L\circ s\equiv uvw\circ L \hbox{ mod } \Lambda$, $L\circ t\equiv s\circ L \hbox{ mod } \Lambda$);

12) $L\circ w\equiv u\circ L \hbox{ mod } \Lambda$, $L\circ u\equiv s\circ L \hbox{ mod } \Lambda$

($\Rightarrow L\circ s\equiv w\circ L \hbox{ mod } \Lambda$, $L\circ t\equiv usu\circ L \hbox{ mod } \Lambda$);

13) $L\circ w\equiv s\circ L \hbox{ mod } \Lambda$, $L\circ u\equiv uvw\circ L \hbox{ mod } \Lambda$

($\Rightarrow L\circ s\equiv v\circ L \hbox{ mod } \Lambda$, $L\circ t\equiv w\circ L \hbox{ mod } \Lambda$);

14) $L\circ w\equiv s\circ L \hbox{ mod } \Lambda$, $L\circ u\equiv w\circ L \hbox{ mod } \Lambda$

($\Rightarrow L\circ s\equiv u\circ L \hbox{ mod } \Lambda$, $L\circ t\equiv uvw\circ L \hbox{ mod } \Lambda$);

15) $L\circ w\equiv usu\circ L \hbox{ mod } \Lambda$, $L\circ u\equiv uvw\circ L \hbox{ mod } \Lambda$

($\Rightarrow L\circ s\equiv u\circ L \hbox{ mod } \Lambda$, $L\circ t\equiv w\circ L \hbox{ mod } \Lambda$);

16) $L\circ w\equiv usu\circ L \hbox{ mod } \Lambda$, $L\circ u\equiv w\circ L \hbox{ mod } \Lambda$

($\Rightarrow L\circ s\equiv v\circ L \hbox{ mod } \Lambda$, $L\circ t\equiv uvw\circ L \hbox{ mod } \Lambda$);

17) $L\circ w\equiv w\circ L \hbox{ mod } \Lambda$, $L\circ u\equiv u\circ L \hbox{ mod } \Lambda$

($\Rightarrow L\circ s\equiv s\circ L \hbox{ mod } \Lambda$, $L\circ t\equiv t\circ L \hbox{ mod } \Lambda$);

18) $L\circ w\equiv w\circ L \hbox{ mod } \Lambda$, $L\circ u\equiv v\circ L \hbox{ mod } \Lambda$

($\Rightarrow L\circ s\equiv usu\circ L \hbox{ mod } \Lambda$, $L\circ t\equiv uvt\circ L \hbox{ mod } \Lambda$);

19) $L\circ w\equiv uvw\circ L \hbox{ mod } \Lambda$, $L\circ u\equiv u\circ L \hbox{ mod } \Lambda$

($\Rightarrow L\circ s\equiv usu\circ L \hbox{ mod } \Lambda$, $L\circ t\equiv t\circ L \hbox{ mod } \Lambda$);

20) $L\circ w\equiv uvw\circ L \hbox{ mod } \Lambda$, $L\circ u\equiv v\circ L \hbox{ mod } \Lambda$

($\Rightarrow L\circ s\equiv s\circ L \hbox{ mod } \Lambda$, $L\circ t\equiv uvt\circ L \hbox{ mod } \Lambda$);

21) $L\circ w\equiv w\circ L \hbox{ mod } \Lambda$, $L\circ u\equiv s\circ L \hbox{ mod } \Lambda$

($\Rightarrow L\circ s\equiv v\circ L \hbox{ mod } \Lambda$, $L\circ t\equiv uts\circ L \hbox{ mod } \Lambda$);

22) $L\circ w\equiv w\circ L \hbox{ mod } \Lambda$, $L\circ u\equiv usu\circ L \hbox{ mod } \Lambda$

($\Rightarrow L\circ s\equiv u\circ L \hbox{ mod } \Lambda$, $L\circ t\equiv ust\circ L \hbox{ mod } \Lambda$);

23) $L\circ w\equiv uvw\circ L \hbox{ mod } \Lambda$, $L\circ u\equiv s\circ L \hbox{ mod } \Lambda$

($\Rightarrow L\circ s\equiv u\circ L \hbox{ mod } \Lambda$, $L\circ t\equiv uts\circ L \hbox{ mod } \Lambda$);

24) $L\circ w\equiv uvw\circ L \hbox{ mod } \Lambda$, $L\circ u\equiv usu\circ L \hbox{ mod } \Lambda$

($\Rightarrow L\circ s\equiv v\circ L \hbox{ mod } \Lambda$, $L\circ t\equiv ust\circ L \hbox{ mod } \Lambda$).

Note that here we have used the relation (\ref{E:A1}), with $\Lambda=L^2(i)$. For instance, in type 1), from $L\circ w\equiv v\circ L \hbox{ mod } L^2(i)$, $L\circ u\equiv uvw\circ L \hbox{ mod } L^2(i)$ and $L\circ s\equiv usu\circ L \hbox{ mod } L^2(i)$, we have
$$A=a\left( \begin{array}{cc} 1 & 1 \\ 1 & -1 \end{array} \right),\ \hbox{ and } \ r\equiv e\left( \begin{array}{c} 1 \\ 1 \end{array} \right) \hbox{ mod } L^2(i),\ \ e\equiv ie \hbox{ mod } L(i).$$
By (\ref{E:A1}), we induce that $a(1+i)\equiv 0 \hbox{ mod } L(i)$. Since
$$L\circ t\equiv \left( a \left( \begin{array}{cc} -1 & -1 \\ 1 & -1 \end{array} \right)\right.\left| e \left( \begin{array}{c} 1 \\ 1 \end{array} \right) + \left( \begin{array}{c} a(1+i) \\ 0 \end{array} \right) \right) \hbox{ mod } L^2(i)$$
$$\hbox{and}\ \ \ w\circ L\equiv \left( a \left( \begin{array}{cc} -1 & -1 \\ 1 & -1 \end{array} \right)\right.\left| e \left( \begin{array}{c} 1 \\ 1 \end{array} \right) \right) \hbox{ mod } L^2(i),$$
we get $L\circ t\equiv w\circ L \hbox{ mod } L^2(i)$. The other types are similar.
\end{proof}

\section{Classification}\label{S:Class}

Let $E$ be an one-dimensional torus and $R_m$ be the group of $m$-th roots of unity acting  on $E$ by rotation around a base point. In \cite{M:Lattes}, it is shown that if $E/R_m\cong \pv^1$ then $m$ is equal to 2, 3, 4 or 6. Denote by $R_m^2$ the group
$$\left\{ \left( \begin{array}{cc} \theta^j & \\ & \theta^k \end{array} \right);\ \theta=e^{2\pi i/m},\ 0\le j,k<m\right\}.$$
Then obviously $E^2/R_m^2\cong \pv^1\times\pv^1$. We have the following

\begin{lem}\label{L:P1}
Let $L=(A|r)$ be an affine map on $E^2$ which induces a Latt\`es map on $\pv^2$ in the case $(\Lambda,G_i)$, $i=1,2,3,4,5$. If $A$ is of the form
$$\left( \begin{array}{cc} \alpha & \\ & \beta \end{array} \right)\ \hbox{ or }\ \left( \begin{array}{cc} & \alpha \\ \beta & \end{array} \right),$$
then $L$ induces a map on $\pv^1\times\pv^1$ of the form $(g,g)$, where $g$ is a Latt\`es map on $\pv^1$.
\end{lem}
\begin{proof}
By Propositions \ref{P:G2}-\ref{P:G5}, we know that $A$ is in fact of the form
$$\left( \begin{array}{cc} a & \\ & \theta^j a \end{array} \right)\ \hbox{ or }\ \left( \begin{array}{cc} & a \\ \theta^j a & \end{array} \right),\ \ \theta=e^{2\pi i/m},\ 0\le j<m,$$
where $m=2$ for $G_1$, $m=3$ for $G_2$, $m=4$ for $G_3$ and $G_5$, and $m=6$ for $G_4$. We also know that
$$r\equiv e\left( \begin{array}{c} 1 \\ 1 \end{array} \right) \hbox{ mod } L^2(\tau),\ \ e\equiv \theta e \hbox{ mod } L(\tau),$$
where $\tau$ is arbitrary for $G_1$ and $\tau=\theta$ for $G_2$, $G_3$, $G_4$ and $G_5$.

Denote by $l$ the affine map on $E$ given by $x\rightarrow ax+e$. By (\ref{E:A1}), we easily see that $aL(\tau)\subset L(\tau)$. And since $e\equiv \theta e \hbox{ mod } L(\tau)$, we can show, as in the proof of \cite[Theorem 3.1]{M:Lattes}, that $l$ induces a Latt\`es map $g$ on $\pv^1$, with deg($g$)$=|a|^2$.

Obviously $l$ and $\theta^j\cdot l$ induce the same map on $\pv^1$. Therefore $L$ induces a map on $\pv^1\times\pv^1$ of the form $(x,y)\rightarrow (g(x),g(y))$ or $(x,y)\rightarrow (g(y),g(x))$.
\end{proof}

By the above lemma, we have the following commutative diagram
$$\xymatrix{
\ E^2\  \ar[d]_L \ar[r]^(.4){\pi_1} & \ \pv^1\times\pv^1 \ar[d]_{(g,g)} \ar[r]^(.6){\pi_2} & \ \pv^2\  \ar[d]^f \\
\ E^2\  \ar[r]^(.4){\pi_1} & \ \pv^1\times\pv^1 \ar[r]^(.6){\pi_2} & \ \pv^2, }$$
where $L$ and $g$ are as in Lemma \ref{L:P1} and $f$ is the Latt\`es map on $\pv^2$ induced by $L$. As mentioned in the introduction, we will call such $f$ the \textit{square map} of $g$.

In \cite{U:Complex}, Ueda showed how to construct a holomorphic map $f$ on $\pv^2$ from a holomorphic map $g$ on $\pv^1$. We will call such $f$ an \textit{Ueda map}. A square map in cases $(\Lambda,G_i)$, $1\le i\le 4$, is an Ueda map, while a square map $f$ in the case $(\Lambda,G_5)$ is semi-conjugate to an Ueda map $h$, that is there exists a non-invertible holomorphic map $\varphi$ on $\pv^2$ such that $f\circ \varphi=\varphi\circ h$. We will explain more details about Ueda's construction and the semi-conjugacy in the next section. Note that an Ueda map is a holomorphic map which preserves an algebraic web associated to a smooth conic (see \cite{DJ:Webs}).

Denote by $A_i$, $1\le i\le 16$, the $i$-th matrix listed in Proposition \ref{P:G5}. It is easy to check that we have

$$A_1^2=A_4^2=A_5^2=A_8^2=2\left( \begin{array}{cc} 1 & \\ & 1 \end{array} \right),$$
$$\begin{array}{ll}
A_2^2=-2\left( \begin{array}{cc} & 1 \\ -1 & \end{array} \right), & A_3^2=2\left( \begin{array}{cc} & 1 \\ -1 & \end{array} \right),\\
A_6^2=-2i\left( \begin{array}{cc} & 1 \\ 1 & \end{array} \right), & A_7^2=2i\left( \begin{array}{cc} & 1 \\ 1 & \end{array} \right),
\end{array}$$
$$\begin{array}{llll}
A_9^2=(1+i)A_{15}, & A_{10}^2=(1-i)A_{13}, & A_{11}^2=(1-i)A_{16}, & A_{12}^2=(1+i)A_{14},\\
A_{13}^2=(1+i)A_{10}, & A_{14}^2=(1-i)A_{12}, & A_{15}^2=(1-i)A_9, & A_{16}^2=(1+i)A_{11},
\end{array}$$
$$A_9^3=A_{12}^3=A_{13}^3=A_{16}^3=2(1+i)\left( \begin{array}{cc} 1 & \\ & 1 \end{array} \right),$$
$$A_{10}^3=A_{11}^3=A_{14}^3=A_{15}^3=2(1-i)\left( \begin{array}{cc} 1 & \\ & 1 \end{array} \right).$$
And it is also easy to check that the form of $r$ is preserved under iterations.

Let $L_i$ be an affine map on $E^2$ of type $i)$, $1\le i\le 16$, as in Proposition \ref{P:G5} which induces a Latt\`es map $f_i$ on $\pv^2$. From the above discussion and Lemma \ref{L:P1}, we see that $f_i^2$ is a square map of a Latt\`es map on $\pv^1$ if $1\le i\le 8$ and that $f_i^3$ is a square map of a Latt\`es map on $\pv^1$ if $9\le i\le 16$. Note that a Latt\`es map in the case $(\Lambda,G_5)$ of types 1)-4) factors through a Latt\`es map in the case $(\Lambda,G_1)$ of the same type.

From Propositions \ref{P:G2}-\ref{P:G5} and the above discussion, we have the following

\begin{thm}\label{T:Lattes}
If $f$ is a Latt\`es map on $\pv^2$, then\\
$i)$ $f$ is a square map of a Latt\`es map on $\pv^1$ in cases $(\Lambda,G_2)$, $(\Lambda,G_3)$ and $(\Lambda,G_4)$, and it is an Ueda map;\\
$ii)$ either $f$ or $f^2$ is a square map of a Latt\`es map on $\pv^1$ in case $(\Lambda,G_1)$, and it is an Ueda map;\\
$iii)$ either $f$, $f^2$ or $f^3$ is a square map of a Latt\`es map on $\pv^1$ in case $(\Lambda,G_5)$, and it is semi-conjegate to an Ueda map.
\end{thm}

To study the case $(\Lambda,G_6)$, we make the following change of coordinates on $E^2$:
$$\left( \begin{array}{c} x \\ y \end{array} \right)=\left( \begin{array}{cc} -1 & \zeta^2 \\ 1 & \zeta \end{array} \right)\left( \begin{array}{c} x^\prime \\ y^\prime \end{array} \right).$$

In the new coordinates $(x^\prime,y^\prime)$, the group $(3,3)_0$ is represented by
$$\left\langle \left( \begin{array}{cc} -1 & -1 \\ 0 & 1 \end{array} \right), \left( \begin{array}{cc} 0 & 1 \\ 1 & 0 \end{array} \right) \right\rangle \ltimes L^2(\tau).$$

In the new coordinates, Proposition \ref{P:G6} can be rephrased as the following

\begin{prop}\label{P:G6-1}
If $L=(A|r)$ induces a Latt\`es map in the case $(\Lambda,G_6)$, then $A=aB$, where $B$ takes one of the following six forms:
$$\begin{array}{lll}
1)\ \left( \begin{array}{cc} 1 & 0 \\ 0 & 1 \end{array} \right), & 2)\ \zeta^{-1}\left( \begin{array}{cc} 1 & 1 \\ -1 & 0 \end{array} \right), & 3)\ \zeta\left( \begin{array}{cc} 0 & -1 \\ 1 & 1 \end{array} \right),\\
4)\ \left( \begin{array}{cc} -1 & -1 \\ 0 & 1 \end{array} \right), & 5)\ \zeta^{-1}\left( \begin{array}{cc} -1 & 0 \\ 1 & 1 \end{array} \right), & 6)\ \zeta\left( \begin{array}{cc} 0 & -1 \\ -1 & 0 \end{array} \right),
\end{array}$$
and
$$r\equiv e\left( \begin{array}{c} 1 \\ 1 \end{array} \right) \hbox{ mod } L^2(\tau),\ \ e\equiv -2e \hbox{ mod } L(\tau).$$
\end{prop}
\begin{proof}
The $A$ part is straightforward. As for $r$, we have
$$\begin{array}{ll}
\left( \begin{array}{c} e \\ f \end{array} \right)\rightarrow \left( \begin{array}{c} e^\prime \\ f^\prime \end{array} \right), & \left( \begin{array}{c} f \\ e \end{array} \right)\rightarrow \left( \begin{array}{c} -(e^\prime+f^\prime) \\ f^\prime \end{array} \right),\\
\left( \begin{array}{c} \zeta^2 f \\ \zeta^{-2} e \end{array} \right)\rightarrow \left( \begin{array}{c} f^\prime \\ e^\prime \end{array} \right), & \left( \begin{array}{c} \zeta^{-2} f \\ \zeta^2 e \end{array} \right)\rightarrow \left( \begin{array}{c} e^\prime \\ -(e^\prime+f^\prime) \end{array} \right).
\end{array}$$
The rest is then easy to check.
\end{proof}

Denote by $B_i$, $1\le i\le 6$, the $i$-th matrix listed in the above proposition. It is obvious that
$$B_1=B_2^3=B_3^3=B_4^2=B_5^6=B_6^6=\left( \begin{array}{cc} 1 & \\ & 1 \end{array} \right),$$
and it is not hard to check that the form of $r$ is preserved under iterations.

Let $L$ be as in Proposition \ref{P:G6-1}. We have that either $L$, $L^2$, $L^3$ or $L^6$ is of the form $(l,l)$, where $l:x\rightarrow ax+e$, with $3e\equiv 0 \hbox{ mod } L(\tau)$, is an affine map on $E$. 

Let $\eta:E\rightarrow \check{\pv}^2$ be an embedding of $E$ into $\check{\pv}^2$ and denote the image by $C$, which is a smooth cubic. Since $3e\equiv 0 \hbox{ mod } L(\tau)$, the map $l$ preserves collinearity. Thus $(l,l)$ induces a holomorphic map on $\pv^2$, through $\phi:E\times E\rightarrow \pv^2$, which preserves the web associated to $C$. Here $\phi$ maps $(x,y)\in E\times E$ to the line joining $\eta(x)$ and $\eta(y)$.

From the above discussion, we have the following

\begin{thm}\label{T:Lattes1}
If $f$ is a Latt\`es map on $\pv^2$ in the case $(\Lambda,G_6)$, then either $f$, $f^2$, $f^3$ or $f^6$ is a holomorphic map preserving an algebraic web associated to a smooth cubic.
\end{thm}

Combining Theorems \ref{T:Lattes} and \ref{T:Lattes1}, we get Theorem \ref{T:Main}.

\section{Elliptic Functions and Parabolic Orbifolds}\label{S:Elliptic}

Recall that $E/R_m\cong \pv^1$ for $m=2,3,4,6$. The map $\pi:E\rightarrow E/R_m$ can be described in terms of classical elliptic function theory. (For backgrounds on elliptic functions, see e.g. \cite{Ch:Elliptic}.)

For $m=2$, we can identify $\pi$ with the Weierstrass function $\wp$ associated to $L(\tau)$. For $m=3$, we can identify $\pi$ with $\wp^\prime$ associated to $L(\zeta)$. For $m=4$, we can identify $\pi$ with $\wp^2$ associated to $L(i)$. For $m=6$, we can identify $\pi$ with $\wp^{\prime 2}$ (or $\wp^3$) associated to $L(\zeta)$. We will denote by $\pi_1$ the map $(\pi,\pi):E^2\rightarrow \pv^1\times \pv^1$.

It is well-known that
$$\wp^{\prime 2}=4\wp^3-g_2\wp-g_3,$$
where $g_2$ and $g_3$ are two invariants of $\wp$, and that
$$\wp(u+v)=\frac{1}{4}\left( \frac{\wp^\prime(u)-\wp^\prime(v)}{\wp(u)-\wp(v)} \right)^2-\wp(u)-\wp(v),$$
where $u\not\equiv v \hbox{ mod } L(\tau)$.

From the above two identities, we have the following two identities:
$$\wp(u+v)+\wp(u-v)=\frac{(2\wp(u)\wp(v)-\frac{1}{2}g_2)(\wp(u)+\wp(v))-g_3}{(\wp(u)-\wp(v))^2},$$
$$\wp(u+v)\wp(u-v)=\frac{(\wp(u)\wp(v)+\frac{1}{4}g_2)^2+g_3(\wp(u)+\wp(v))}{(\wp(u)-\wp(v))^2}.$$

For $L(i)$, we have $g_3=0$. For simplicity, we will also choose $g_2=4$. Then
\begin{equation}\label{E:E1}
\wp^{\prime 2}=4(\wp^3-\wp)
\end{equation}
and for $u\not\equiv v \hbox{ mod } L(i)$, we have
\begin{equation}\label{E:E2}
\wp(u+v)+\wp(u-v)=\frac{2(\wp(u)\wp(v)-1)(\wp(u)+\wp(v))}{(\wp(u)-\wp(v))^2},
\end{equation}
\begin{equation}\label{E:E3}
\wp(u+v)\wp(u-v)=\frac{(\wp(u)\wp(v)+1)^2}{(\wp(u)-\wp(v))^2}.
\end{equation}

We now recall Ueda's construction in \cite{U:Complex}. Let $g$ be a holomorphic map on $\pv^1$. A holomorphic map $f$ on $\pv^2$ is constructed from $g$ as follows. Let $C$ be a smooth conic in $\pv^2$, which we identify with $\pv^1$. For a point $P\in \pv^2$, let $l_1$ and $l_2$ be the tangent lines to $C$ which pass through $P$ and let $Q_1$ and $Q_2$ be the points of contact. (Let $Q_1=Q_2=P$ if $P\in C$.) Let $l_1^\prime$ and $l_2^\prime$ be the tangent lines to $C$ at $g(Q_1)$ and $g(Q_2)$ and define $f(P)$ to be the intersection point of $l_1^\prime$ and $l_2^\prime$. (Let $f(P)=f(Q_1)=f(Q_2)$ if $f(Q_1)=f(Q_2)$.) It is easy to see that $f$ preserves the algebraic web associated to the dual curve of $C$.

Denote by $\iota$ the automorphism of $\pv^1\times \pv^1$ defined by $\iota:(Q_1,Q_2)\rightarrow (Q_2,Q_1)$. Then we have $(\pv^1\times \pv^1)/\langle\iota\rangle\cong \pv^2$. The map $\pi_2:\pv^1\times \pv^1\rightarrow \pv^2$ is a double cover having a smooth conic $C=\pi_2(\Delta)$ as the branch locus, where $\Delta=\{(Q_1,Q_2)\in \pv^1\times \pv^1|Q_1=Q_2\}$. More explicitly, the map
$$\pi_2:([\xi:\eta],[\xi^\prime:\eta^\prime])\in \pv^1\times \pv^1\rightarrow [x:y:z]\in \pv^2$$
is given by a triple of bilinear symmetric forms of $(\xi,\eta)$ and $(\xi^\prime,\eta^\prime)$. Here we choose
\begin{equation}\label{E:Pi2}
[x:y:z]=\pi_2([\xi:\eta],[\xi^\prime:\eta^\prime])=[\xi\eta^\prime+\xi^\prime\eta:\xi\xi^\prime:\eta\eta^\prime].
\end{equation}
Then $C=\pi_2(\Delta)=\{x^2-4yz=0\}$.

In cases $(\Lambda,G_i)$, $i=1,2,3,4$, the map $\Psi:E^2\rightarrow \pv^2$ is the composition $\pi_2\circ \pi_1$. Denote by $D\subset \pv^2$ the branch locus of $\Psi$. For each irreducible component $D_j$ of $D$, we associate with a pair $(d_j,r_j)$, where $d_j$ is the degree of $D_j$ and $r_j$ is the ramification index of $\Psi$ along $D_j$. Then for cases $(\Lambda,G_i)$, $i=1,2,3,4$, we have the following list of $(d_j,r_j)$ (see \cite{KTY:II}):
\begin{equation}\label{E:D1}
\begin{array}{ll}
(\Lambda,G_1)\ [m=2]: & \{(1,2),(1,2),(1,2),(1,2),(2,2)\}\\
(\Lambda,G_2)\ [m=3]: & \{(1,3),(1,3),(1,3),(2,2)\}\\
(\Lambda,G_3)\ [m=4]: & \{(1,2),(1,4),(1,4),(2,2)\}\\
(\Lambda,G_4)\ [m=6]: & \{(1,2),(1,3),(1,6),(2,2)\}
\end{array}
\end{equation}

Let $\varphi:\pv^2\rightarrow \pv^2$ be the map $[x:y:z]\mapsto [x^2:(y+\alpha z)^2:(y-\alpha z)^2]$, $\alpha\in \cv^\star$. Then $C=\{x^2-4yz=0\}$ is mapped to $\{\alpha x-y+z=0\}$. For simplicity, we choose $\alpha=1$. Thus, we have
\begin{equation}\label{E:Pi3}
\varphi:[x:y:z]\rightarrow [x^2:(y+z)^2:(y-z)^2]\ \hbox{ and } \varphi(C)=\{x-y+z=0\}.
\end{equation}
The map $\varphi$ provides the semi-conjugacy mentioned in the last section. In the case $(\Lambda,G_5)$, the map $\Psi:E^2\rightarrow \pv^2$ is the composition $\varphi\circ \pi_2\circ \pi_1$, and we have the following list of $(d_j,r_j)$ (see \cite{KTY:II}):
\begin{equation}\label{E:D2}
(\Lambda,G_5):\ \ \{(1,2),(1,2),(1,2),(1,2),(1,2),(1,2)\}.
\end{equation}

As we have seen in the last section, the case $(\Lambda,G_6)$ is quite different. Embed $E$ into $\check{\pv}^2$ by $\eta:u\rightarrow [\wp(u):\wp^\prime(u):1]$. Since $(u,v)\in E^2$ corresponds to the line joining $\eta(u)$ and $\eta(v)$, the map $\Psi:E^2\rightarrow \pv^2$ is given by
\begin{equation}\label{E:Psi}
\Psi:(u,v)\in E^2\rightarrow [\wp^\prime(v)-\wp^\prime(u):\wp(u)-\wp(v):\wp^\prime(u)\wp(v)-\wp(u)\wp^\prime(v)].
\end{equation}
The branch locus $D$ of $\Psi$ is the dual curve of $\eta(E)\subset \check{\pv}^2$, which is irreducible and of degree 6. The ramification index of $\Psi$ along $D$ is 2 (see \cite{KTY:II}). Therefore, we have
\begin{equation}\label{E:D3}
(\Lambda,G_6):\ \ \{(6,2)\}.
\end{equation}

A holomorphic map $f$ on $\pv^k$ is said to be \textit{critically finite}, if every irreducible component of the critical set $C_f$ of $f$ is periodic or pre-periodic. (For dynamical properties of such maps, see e.g. \cite{U:Critical},\cite{J:Critical},\cite{R:Critical}.) It is not hard to show that Latt\`es maps are critically finite (see \cite[Lemma 3.4]{M:Lattes},\cite[Lemma 5.1]{Du:Lattes}). Moreover, the post-critical set $V_f$ of a Latt\`es map $f$ is precisely equal to the branch locus $D$ of the map $\Psi:E^k\rightarrow \pv^k$.

We now introduce the notion of ``orbifold", as considered by Thurston (\cite{T:Comb}). (Note that our notion here is more restrictive than the general notion.)

Let $X$ be a complex manifold and denote by $\mathcal{H}(X)$ the space of irreducible analytic subvarieties of codimension 1 in $X$. Let $r$ be a function defined on $\mathcal{H}(X)$ with values in $\mathbf{N}^+$, which is equal to 1 outside a locally finite family of analytic subvarieties of $X$. We call the pair $(X,r)$ an \textit{orbifold}. We say that the orbifold $(X,r)$ is \textit{parabolic}, if there exists a ramified covering $f:X\rightarrow X$ such that
\begin{equation}\label{E:Parabolic}
r(f(H))=m_f(H)\cdot r(H)
\end{equation}
for every $H\in \mathcal{H}(X)$, where $m_f(H)$ is the multiplicity of $f$ along $H$ (i.e. the multiplicity of $f$ at a generic point of $H$). (Note that $m_f(H)=1$ if $H\not\subset C_f$.) We have the following

\begin{lem}\label{L:Orbifold}
If $f$ is a Latt\`es map on $\pv^k$, then there exists a parabolic orbifold $(\pv^k,r)$ satisfying (\ref{E:Parabolic}).
\end{lem}
\begin{proof}
Assume that $f$ is induced by $L:E^k\rightarrow E^k$, through $\Psi:E^k\rightarrow \pv^k$. Let $D$ be the ramification locus of $\Psi$ and denote by $r_\Psi(H)$ the ramification index of $\Psi$ along an irreducible component $H$ of $D$. We define the function $r$ as follows:
$$r(H)=\left\{ \begin{array}{ll} r_\Psi(H), & \ \ H\subset D;\\ 1, & \ \ H\not\subset D. \end{array} \right.$$

For any $H$, choose generic $z_0\in H$, $z_1\in f(H)$ and $\upsilon_0,\upsilon_1\in E^k$ such that the following diagram commutes:
$$\xymatrix{
\ \upsilon_0\  \ar@{|->}[d]_\Psi \ar@{|->}[r]^L & \ \upsilon_1\  \ar@{|->}[d]^\Psi \\
\ z_0\  \ar@{|->}[r]^f & \ z_1. }$$

Denote by $d_\phi(z)$ the local degree of a map $\phi$ at a point $z$. Then we have $d_L(\upsilon)=1$ for every $\upsilon\in E^k$. Therefore, we get
$$d_\Psi(\upsilon_1)=d_f(z_0)\cdot d_\Psi(\upsilon_0).$$
By the definitions of $m_f(H)$ and $r(H)$, we then have (\ref{E:Parabolic}).
\end{proof}

\section{Examples}\label{S:Examples}

In this section, we give some typical examples of Latt\`es maps in cases $(\Lambda,G_i)$, $1\le i\le 5$. For each map, we will also give the ``orbifold portrait", which describes how each irreducible critical component is mapped under iterations along with the ramification index of each irreducible post-critical component.

We first give several examples of Ueda maps. If $f$ is an Ueda map induced by a Latt\`es map $g$ on $\pv^1$, then the orbifold portrait of $f$ is decided by the orbifold portrait of $g$. Let $C$ be the smooth conic in the Ueda construction and denote by $l_p$ the tangent line at $p\in C$. Then we know that $l_{g(p)}=f(l_p)$. In our chosen coordinates (\ref{E:Pi2}), $l_p$ at $p=[c:1]$ is given by $\{cx-y-c^2z=0\}$ and $l_p$ at $p=[1:0]$ is given by $\{z=0\}$.

In the case $(\Lambda,G_1)$, from $g:[\xi:\eta]\rightarrow [\xi^2+\eta^2:2i\xi\eta]$ (when $a=1+i$), we get
\begin{equation}\label{E:f1}
f:[x:y:z]\longrightarrow [2ix(y+z):x^2+(y-z)^2:-4yz],
\end{equation}
with the orbifold portrait:
$$\xymatrix{
\ \ \{y-z=0\} \ar[r] & *{\ \{x^2-4yz=0\}(2)\ \ \ \ \ \ \ \ \ \ \ \ \ \ \ \ \ \ \ \ \ \ } \ar@(ur,dr) & &\ \ \ \ \ \ }$$
$$\xymatrix@R=2pt{
\{x-y-z=0\} \ar[r] & \{x-iy+iz=0\}(2) \ar[dr]!(0,-1) \\
& & \{y=0\}(2) \ar[r] & *{\ \{z=0\}(2),} \ar@<-0.5ex>@(dr,dl) \\
\{x+y+z=0\} \ar[r] & \{x+iy-iz=0\}(2) \ar[ur]!(0,1) }$$
where $(\cdot)$ is the ramification index of the map along each irreducible post-critical component.

In the case $(\Lambda,G_2)$, from $g:[\xi:\eta]\rightarrow [i(\xi^3+3\xi\eta^2):\sqrt{3}(3\xi^2\eta+\eta^3)]$ (when $a=\sqrt{3}i$), we get
$$f:[x:y:z]\longrightarrow [\sqrt{3}ix(x^2+3(y+z)^2):-y(3x^2+(y-3z)^2):3z(3x^2+(3y-z)^2)],$$
with the orbifold portrait:
$$\xymatrix{
\{x^2+(3y-z)(y-3z)=0\} \ar[r] & *{\ \{x^2-4yz=0\}(2)\ \ \ \ \ \ \ \ \ \ \ \ \ \ \ \ \ \ \ \ \ \ } \ar@(ur,dr) &\ \ \ \ \ \ }$$
$$\xymatrix@R=2pt{
\ \ \ \ \ \ \ \ \ \ \{x-y-z=0\} \ar[r] & \{\sqrt{3}ix-3y+z=0\}(3) \ar[dr]!(0,-1) \\
& & *{\ \{z=0\}(3).} \ar@<-0.5ex>@(dr,dl) \\
\ \ \ \ \ \ \ \ \ \ \{x+y+z=0\} \ar[r] & \{\sqrt{3}ix+3y-z=0\}(3) \ar[ur]!(0,1) }$$

In the case $(\Lambda,G_3)$, from $g:[\xi:\eta]\rightarrow [(\xi+\eta)^2:-4\xi\eta]$ (when $a=1+i$), we get
$$f:[x:y:z]\longrightarrow [-4(x(y+z)+4yz):x^2+y^2+z^2+2(xy+yz+xz):16yz],$$
with the orbifold portrait:
$$\xymatrix{
\ \ \{y-z=0\} \ar[r] & *{\ \{x^2-4yz=0\}(2)\ \ \ \ \ \ \ \ \ \ \ \ \ \ \ \ \ \ \ \ \ \ } \ar@(ur,dr) & &\ \ \ \ }$$
$$\xymatrix{
\{x-y-z=0\} \ar[r] & \{x+y+z=0\}(2) \ar[r] & \{y=0\}(4) \ar[r] & *{\ \{z=0\}(4).} \ar@<-0.5ex>@(dr,dl) }$$

\vspace{4mm}

In the case $(\Lambda,G_4)$, from $g:[\xi:\eta]\rightarrow [-\xi(\xi+3\eta)^2:3\eta(3\xi+\eta)^2]$ (when $a=\sqrt{3}i$), we get
\begin{flushleft}
$f:[x:y:z]\longrightarrow$
\end{flushleft}
\begin{flushright}
$[-3(x(x+3(y+z))^2+96(x+y+z)yz):y(3x+y+9z)^2:9z(3x+9y+z)^2],$
\end{flushright}
with the orbifold portrait:
$$\xymatrix{
\{(x+3(y+z))^2-64yz=0\} \ar[r] & *{\ \{x^2-4yz=0\}(2)\ \ \ \ \ \ \ \ \ \ \ \ \ \ \ \ \ \ \ \ \ \ } \ar@(ur,dr) & }$$
$$\xymatrix{
\ \ \ \ \ \ \ \ \ \ \ \ \ \{x-y-z=0\} \ar[r] & \{3x+9y+z=0\}(3) \ar[r] & *{\ \{z=0\}(6)} \ar@<-0.5ex>@(dr,dl) }$$
$$\xymatrix{
\ \ \ \ \{3x+y+9z=0\} \ar[r] & *{\ \{y=0\}(2).} \ar@<-0.5ex>@(dr,dl) & &\ \ \ \ \ \ \ \ \ \ \ \ }$$

\vspace{4mm}

Back to the case $(\Lambda,G_1)$, we now give an example of Latt\`es map $f$ such that $f^2$ is an Ueda map. The map is of type 1) with $a=1$. It is obtained by expressing $\wp(u+v)+\wp(u-v)$ and $\wp(u+v)\wp(u-v)$ in terms of $\wp(u)+\wp(v)$ and $\wp(u)\wp(v)$, using (\ref{E:E2}) and (\ref{E:E3}). We have
\begin{equation}\label{E:f2}
f:[x:y:z]\longrightarrow [2x(y-z):(y+z)^2:x^2-4yz],
\end{equation}
with the orbifold portrait:
$$\xymatrix@R=3pt@C-5pt{
\ \ \ \ \ \{y+z=0\} \ar[r] & *{\ \{y=0\}(2)} \ar[dr]!(0,-2) \\
\{x+y-z=0\} \ar[r] & \{x+y+z=0\}(2) \ar[r] & \{x^2-4yz=0\}(2) \ar@/^/[r]!(0,-2) & \{z=0\}(2). \ar@/^/[l]!(0,2) \\
\{x-y+z=0\} \ar[r] & \{x-y-z=0\}(2) \ar[ur]!(0,2) }$$
One readily checks that $f^2$ is an Ueda map with the orbifold portrait:
$$\xymatrix@R=3pt{
\ \ \ \ \ \ \ \ \ \ \ \ \ \ \ \ \ \ \ \ \ \ \{y+z=0\} \ar[dr]!(-10,-1.5) \\
\ \ \ \ \ \ \ \ \ \ \ \ \ \ \ \ \ \{x+y-z=0\} \ar[r] & *{\ \{x^2-4yz=0\}(2)\ \ \ \ \ \ \ \ \ \ \ \ \ \ \ \ \ \ \ \ \ \ } \ar@(ur,dr) \\
\ \ \ \ \ \ \ \ \ \ \ \ \ \ \ \ \ \{x-y+z=0\} \ar[ur]!(-10,1.5) }$$
$$\xymatrix@R=1pt{
\ \ \ \ \ \ \ \ \ \ \ \ \ \ \ \ \{x+i(y-z)=0\} \ar[dr]!(0,-0.5) \\
& \{y=0\}(2) \ar[dddr]!(-9,1) \\
\ \ \ \ \ \ \ \ \ \ \ \ \ \ \ \ \{x-i(y-z)=0\} \ar[ur]!(0,0.5) \\
\{x-(y-z)+\sqrt{2}(y+z)=0\} \ar[dr]!(0,-2) \\
& \{x+y+z=0\}(2) \ar[r] & *{\ \{z=0\}(2).} \ar@<-0.5ex>@(dr,dl) \\
\{x-(y-z)-\sqrt{2}(y+z)=0\} \ar[ur]!(0,2) \\
\{x+(y-z)+\sqrt{2}(y+z)=0\} \ar[dr]!(0,-2) \\
& \{x-y-z=0\}(2) \ar[uuur]!(-9,-1) \\
\{x+(y-z)-\sqrt{2}(y+z)=0\} \ar[ur]!(0,2) }$$

Through the semi-conjugation (\ref{E:Pi3}), maps (\ref{E:f1}) and (\ref{E:f2}) induce Latt\`es maps in the case $(\Lambda,G_5)$. From (\ref{E:f1}), we get
\begin{equation}\label{E:f3}
f:[x:y:z]\longrightarrow [-4xy:(x-y+2z)^2:(x+y)^2],
\end{equation}
with the orbifold portrait:
$$\xymatrix{
\ \ \ \ \ \ \ \ \{x+y=0\} \ar[r] & \{z=0\}(2) \ar[r] & *{\ \{x-y+z=0\}(2)\ \ \ \ \ \ \ \ \ \ \ \ \ \ \ \ \ \ \ \ \ \ \ } \ar@(ur,dr) }$$
$$\xymatrix{
\ \ \ \{x-y=0\} \ar[r] & \{x+z=0\}(2) \ar[r] & *{\ \{y-z=0\}(2)\ \ \ \ \ \ \ \ \ \ \ \ \ \ \ \ \ \ } \ar@(ur,dr) }$$
$$\xymatrix{
\{x-y+2z=0\} \ar[r] & \{y=0\}(2) \ar[r] & *{\ \{x=0\}(2).} \ar@<-0.5ex>@(dr,dl) & \ \ \ \ \ \ \ \ \ \ \ \ \ \ \ \ \ \ \ \ \ \ }$$

\vspace{4mm}

From (\ref{E:f2}), we get
\begin{equation}\label{E:f4}
f:[x:y:z]\longrightarrow [4xz:(x+z)^2:(x-2y+z)^2],
\end{equation}
with the orbifold portrait:
$$\xymatrix@R=6pt{
\ \ \ \ \ \ \ \{x+z=0\} \ar[r] & \{y=0\}(2) \ar[r] & \{y-z=0\}(2) \ar@/^5pc/[d] \\
\ \ \ \ \ \ \ \{x-z=0\} \ar[r] & \{x-y=0\}(2) \ar[r] & \{x-y+z=0\}(2) \ar@/^5pc/[u] \\
\{x-2y+z=0\} \ar[r] & \{z=0\}(2) \ar[r] & *{\ \{x=0\}(2).} \ar@<-0.5ex>@(dr,dl) }$$

\vspace{4mm}

Note that maps (\ref{E:f3}) and (\ref{E:f4}) are conjugate to $f_1$ and $f_2$ in \cite[Proposition 5.2]{Du:Lattes}, respectively. And $f_3$ in \cite[Proposition 5.2]{Du:Lattes} is of type 13) with $a=1$ (thus $f_3^3$ is semi-conjugate to an Ueda map).

\end{document}